\documentclass[12pt]{amsart}
\usepackage{amsmath, amssymb, amsfonts}
\usepackage[all]{xypic}
\usepackage[pdftex]{graphicx}
\usepackage[usenames]{color}
\usepackage{stmaryrd}
\usepackage[utf8]{inputenc}
\usepackage{bbold}

\theoremstyle{plain}
\newtheorem{theorem}{Theorem}[section]
\newtheorem{lemma}[theorem]{Lemma}
\newtheorem{corollary}[theorem]{Corollary}
\newtheorem{proposition}[theorem]{Proposition}
\newtheorem{conjecture}[theorem]{Conjecture}

\theoremstyle{remark}
\newtheorem{remark}[theorem]{Remark}

\newtheorem{example}[theorem]{Example}

\newtheorem*{note*}{Note}
\newtheorem*{remark*}{Remark}
\newtheorem*{example*}{Example}

\theoremstyle{definition}
\newtheorem*{definition*}{Definition}
\newtheorem{definition}[theorem]{Definition}

\newcommand{\cl}{\mathrm{cl}}
\newcommand{\Ann}{\mathrm{Ann}}
\newcommand{\ram}{\mathrm{ram}}
\newcommand{\Gal}{\mathrm{Gal}}
\newcommand{\Hom}{\mathrm{Hom}}
\newcommand{\Aff}{\mathrm{Aff}}

\newcommand{\Aut}{\mathrm{Aut}}

\newcommand{\exend}{\hfill$\Diamond$}
\newcommand{\hyp}{\mathsf{Hyp}}
\DeclareMathOperator{\Irr}{Irr}
\DeclareMathOperator{\Quot}{Quot}
\DeclareMathOperator{\nr}{nr} 

\newcommand{\ZZ}{\mathbb{Z}}
\newcommand{\RR}{\mathbb{R}}
\newcommand{\QQ}{\mathbb{Q}}
\newcommand{\CC}{\mathbb{C}}
\newcommand{\NN}{\mathbb{N}}
\newcommand{\FF}{\mathbb{F}}

\renewcommand{\det}{\mathrm{det}}
\numberwithin{equation}{section}

\title[Conjectures of Brumer, Gross and Stark]{Conjectures of Brumer, Gross and Stark}

\author{Andreas Nickel}
\address{Universit\"{a}t Duisburg--Essen\\
	Fakult\"{a}t f\"{u}r Mathematik\\
	Thea-Leymann-Str. 9\\
	45127 Essen\\
	Germany}
\email{andreas.nickel@uni-due.de}
\urladdr{https://www.uni-due.de/$\sim$hm0251/english.html}

\subjclass[2010]{11R42, 11R29, 11R23}
\keywords{Stark conjectures, Brumer's conjecture, $L$-values, class groups,
	equivariant Tamagawa number conjecture, main conjecture, Iwasawa theory}
\date{Version of 10th October 2017}
\begin{document}
	
\begin{abstract}
We give an introduction to generalisations of conjectures of Brumer and
Stark on the annihilator of the class group of a number field. We review the 
relation to the equivariant Tamagawa number conjecture, the main conjecture 
of Iwasawa theory for totally real fields, and a conjecture of Gross on the 
behaviour of $p$-adic Artin $L$-functions at zero.
\end{abstract}

\maketitle

\section*{Introduction}

Let $L/K$ be a finite Galois extension of number fields with Galois
group $G$. To each finite set $S$ of places of $K$ containing
all the archimedean places, one can associate a so-called
`Stickelberger element' $\theta_S$ in the centre of the group algebra
$\CC[G]$. This element is constructed from values at $s=0$ of 
the $S$-truncated Artin $L$-series attached to the complex
irreducible characters of $G$. In particular, $\theta_S$ is analytic in nature.
By a result of Siegel \cite{c7-Si70} one knows that
$\theta_S$ always has rational coefficients.

Let $\mu_L$ and $\cl_L$ be the roots of unity and the class group 
of $L$, respectively. Suppose that $S$ also contains all places of $K$
which ramify in $L/K$. Then it was independently shown by
Cassou-Nogu\`es \cite{c7-CN79}, Deligne--Ribet \cite{c7-DR80}
and Barsky \cite{c7-Ba78} that for abelian $G$ one has
\[
\Ann^{}_{\ZZ[G]}(\mu^{}_L)\theta^{}_S\, \subseteq\, \ZZ[G],
\]
where $\Ann_R(M)$ denotes the annihilator ideal of $M$
regarded as a module over the ring $R$.
In other words, the coefficients of $\theta_S$ are almost integral.
Now Brumer's conjecture simply asserts that
$\Ann_{\ZZ[G]}(\mu_L)\theta_S$ annihilates $\cl_L$.
In the case $K = \QQ$ Brumer's conjecture is just Stickelberger's
theorem from the late 19th century \cite{c7-St90}.
Roughly speaking, the conjecture predicts that an analytic object 
gives constraints on the structure of an arithmetic object.
It is this kind of conjecture which is often called a
`Stark-type conjecture'. 

In fact, Harold Stark suggested the following refinement of
Brumer's conjecture. Let $w_L$ be the cardinality of $\mu_L$
and fix a fractional ideal $\mathfrak a$ in $L$.
We will denote the action of $G$ on $\mathfrak a$ 
and on its class in $\cl_L$ by exponents
on the right as usual.
Then the Brumer--Stark conjecture not only predicts that
$\mathfrak a^{w_L \theta_S}$ becomes principal, but also
gives precise information about a generator of that ideal.

In this survey article we explain recent generalisations 
of these conjectures to arbitrary, not necessarily abelian
Galois extensions; these are due to the author \cite{c7-Ni11}
and, independently and in even greater generality, due to
Burns \cite{c7-Bu11}. A further, slightly different approach has been 
developed by Dejou and Roblot \cite{c7-DR14}.

We discuss the relation to further conjectures in the field
such as the equivariant Tamagawa number conjecture of
Burns and Flach \cite{c7-BF01} and the main conjecture
of equivariant Iwasawa theory which (under a suitable condition)
has been proven by Ritter--Weiss \cite{c7-RW11} and, independently,
by Kakde \cite{c7-Ka13}. A conjecture of Gross \cite{c7-Gr81}
on the behaviour of $p$-adic Artin $L$-series at $s=0$ also
plays a pivotal role. 

Roughly speaking, the latter conjecture
asserts that (i) the order of vanishing at $s=0$ of the $p$-adic
$L$-series coincides
with the order of vanishing at $s=0$ of a corresponding complex
$L$-series and (ii) the special values of the $p$-adic and the 
complex $L$-series at $s=0$ coincide up to some explicit $p$-adic regulator.
Considerable progress has been made in the recent years by
Spiess \cite{c7-Sp14} and Burns \cite{c7-Bu17} on part (i),
and by Dasgupta, Kakde and Ventullo \cite{c7-DKV17} on part (ii).
We also explain their results and the relation to the
equivariant Tamagawa number conjecture due to Burns \cite{c7-Bu17} 
as well as the relation to the (non-abelian)
Brumer--Stark conjecture due to Johnston and the author \cite{c7-JN17b}.

We provide no proofs unless they are short and we feel that it might help
for a better understanding. Instead we include some examples
to illustrate occurring obstacles and the underlying ideas
how to overcome them.

\subsection*{Notation and conventions}
All rings are assumed to have an identity element and all modules are assumed
to be left modules unless otherwise  stated.
Unadorned tensor products will always denote tensor 
products over $\ZZ$.
We fix the following notation:

\medskip

\begin{tabular}{ll}
	$R^{\times}$ & the group of units of a ring $R$\\
	$\zeta(R)$ & the centre of a ring $R$\\
	$\Ann_{R}(M)$ & the annihilator ideal of the $R$-module $M$\\
	$M_{m \times n} (R)$ & the set of all $m \times n$ matrices with entries in a ring $R$\\
	%	$\zeta_{n}$ & a primitive $n$th root of unity\\
	$K_{\infty}$ & the cyclotomic $\ZZ_{p}$-extension of the number field $K$\\
	$\mu_{K}$ & the roots of unity of a field $K$\\
	$\cl_{K}$ & the class group of a number field $K$ \\
	$K^{c}$ & an algebraic closure of a field $K$ \\
	%	$K^{+}$ & the maximal totally real subfield of a field $K$ embeddable into $\CC$\\
	$\Irr_{F}(G)$ & the set of $F$-irreducible characters of the (pro)-finite group $G$\\
	& (with open kernel) where $F$ is a field of characteristic $0$ %\\
	%	$\check{\chi}$ & the character contragredient to a character $\chi$
\end{tabular}

\section{Preliminaries}

\subsection{Ray class groups} %\label{c7-subsec:ray-class-groups}

Let $L/K$ be a finite Galois extension of number fields with Galois group $G$.
For each place $v$ of $K$ we fix a place $w$ of $L$ above $v$ and write $G_{w}$ and $I_{w}$ for the decomposition
group and inertia subgroup of $L/K$ at $w$, respectively.
When $w$ is a finite place, we  choose a lift $\phi_{w} \in G_{w}$ of the Frobenius automorphism at $w$;
moreover, we write $\mathfrak{P}_{w}$ for the associated prime ideal in $L$ and $| \cdot |_w$ for the
corresponding absolute value. We denote the cardinality 
of the residue field of $K$ at $v$ by $N(v)$.

For any set $S$ of places of $K$, we write $S(L)$ for the set of places of $L$ which lie above those in $S$.
Now let $S$ be a finite set of places of $K$ containing the set $S_{\infty}=S_{\infty}(K)$ of archimedean places
and let $T$ be a second finite set of places of $K$ such that $S \cap T = \emptyset$.
We write $\cl_{L}^{T}$ for the ray class group 
\index{class group!ray class group}
of $L$ associated to the ray
$\mathfrak{M}_{L}^{T} := \prod_{w \in T(L)} \mathfrak{P}_w$ and $\mathcal{O}_{L,S}$ for the ring of $S(L)$-integers in $L$.
Let $\mathcal{O}_{L} := \mathcal{O}_{L, S_{\infty}}$ be the ring of integers in $L$.
Let $S_{f}$ be the set of all finite places in $S$;
then there is a natural map $\ZZ [S_{f}(L)] \to \cl_{L}^{T}$ which sends each place $w \in S_{f}(L)$
to the corresponding class $[\mathfrak{P}_w] \in \cl_{L}^{T}$. We denote the cokernel of this map by $\cl_{L,S}^{T}$.
When $T$ is empty, we abbreviate $\cl_{L,S}^{\emptyset}$ to $\cl_{L,S}$
so that in particular $\cl_{L, \emptyset} = \cl_L$ is the usual class
group of $L$.
Moreover, we denote the $S(L)$-units of $L$ by $E_{L,S}$ and define
$E_{L,S}^T := \left\{x \in E_{L,S}: x \equiv 1 \bmod \mathfrak{M}_{L}^{T} \right\}$.
%    If $S = S_{\infty}$, we also write $E_L^T$ for $E_{L,S_{\infty}}^T$.
All these modules are equipped with a natural $G$-action and we have an
exact sequence of $\ZZ[G]$-modules
%If $\Sigma$ is a subset of $S$ containing $S_{\infty}$, then we have
%\begin{equation}\label{c7-eqn:ray_class_sequence_ZS}
%0 \longrightarrow E_{L, \Sigma}^T \longrightarrow E_{L,S}^T \stackrel{v_{L}}{\longrightarrow}
%\ZZ [S(L) - \Sigma(L)] \longrightarrow \cl_{L, \Sigma}^{T} \longrightarrow \cl_{L,S}^{T} \longrightarrow 0,
%\end{equation}
%where $v_{L}(x) := \sum_{w \in S(L) - \Sigma(L)} \ord_w(x) w$ for every $x \in E_{L,S}^T$, and
\[
0 \longrightarrow E_{L,S}^T \longrightarrow E_{L,S} \longrightarrow (\mathcal{O}_{L,S} / \mathfrak{M}_{L}^{T})^{\times}
\stackrel{\nu}{\longrightarrow} \cl_{L,S}^{T} \longrightarrow \cl_{L,S} \longrightarrow 0,
\]
where the map $\nu$ lifts an element $\overline x \in (\mathcal{O}_{L,S} / \mathfrak{M}_{L}^{T})^{\times}$ to
$x \in \mathcal{O}_{L,S}$ and
sends it to the ideal class $[(x)] \in \cl_{L,S}^{T}$ of the principal ideal $(x)$.

\subsection{Equivariant Artin $L$-values} %\label{c7-subsec:L-values}

Let $S$ be a finite set of places of $K$ containing $S_{\infty}$.
Let $\Irr_{\CC}(G)$ denote the set of complex irreducible characters of $G$.
For $\chi \in \Irr_{\CC}(G)$, let $V_{\chi}$ be a left $\CC[G]$-module with character $\chi$.
We write $L_{S}(s,\chi)$ for the $S$-truncated Artin $L$-function 
\index{Artin $L$-function!complex}
attached to $\chi$ which for $\mathrm{Re}(s) > 1$ is given by
the Euler product
\[
L^{}_S(s,\chi)\, = \prod_{v \not\in S} \det \left(1 - N(v)^{-s} \phi^{}_w \mid
V_{\chi}^{I_w}\right)^{-1}.
\]
Each element in $\CC[G]$ may be viewed as a complex valued function on $G$.
The irreducible characters constitute a basis of the centre and we thus have
a canonical isomorphism
$\zeta(\CC[G]) \simeq \prod_{\chi \in \Irr_{\CC} (G)} \CC$.
We define the equivariant $S$-truncated Artin $L$-function 
\index{Artin $L$-function!equivariant}
to be the meromorphic $\zeta(\CC[G])$-valued function
\[
L^{}_{S}(s)\, :=\, \left(L^{}_{S}(s,\chi)\right)^{}_{\chi \in \Irr_{\CC} (G)}.
\]
If $T$ is a second finite set of places of $K$ such that $S \cap T = \emptyset$, we define
\[
\delta^{}_{T}(s,\chi)\, = \prod_{v \in T} \det\left( 1 - N(v)^{1-s} \phi_{w}^{-1} \mid V_{\chi}^{I_{w}}\right)
\]
and
\[
\delta^{}_{T}(s)\, :=\, \left(\delta^{}_{T}(s,\chi)\right)^{}_{\chi\in \Irr^{}_{\CC} (G)}.
\]
We set
\[
\Theta^{}_{S,T}(s)\, :=\, \delta^{}_{T}(s) \cdot L^{}_{S}(s)^{\sharp},
\]
where $^{\sharp}: \CC[G] \to \CC[G]$ denotes the anti-involution induced by $g \mapsto g^{-1}$ for $g \in G$.
Note that $L_{S}(s)^{\sharp} = (L_{S}(s,\check{\chi}))_{\chi \in \Irr_{\CC} (G)}$ 
where $\check \chi$ denotes the character contragredient to $\chi$.
The functions $\Theta_{S,T}(s)$ are the so-called $(S,T)$-modified
$G$-equi\-variant 
$L$-functions, and we define Stickelberger elements
\index{Stickelberger element}
\[
\theta_{S}^{T}(L/K)\, =\, \theta_{S}^{T}\, :=\, \Theta^{}_{S,T}(0) \in \zeta(\QQ[G]).
\]
Note that a priori we only have $\theta_{S}^{T} \in \zeta(\CC[G])$, but by a result of Siegel \cite{c7-Si70} we know that $\theta_{S}^{T}$
has rational coefficients.
If $T$ is empty, we abbreviate $\theta_{S}^{T}$ to $\theta_{S}$.

\begin{remark} \label{c7-rem:vanishing}
	Let $\chi \in \Irr^{}_{\CC}(G)$ and let $r^{}_S(\chi)$ be the order of vanishing 
	of $L^{}_{S}(s,\chi)$ at $s=0$.\index{order of vanishing}
	Then by \cite[Chapitre I, Proposition 3.4]{c7-Ta84} one has
	\begin{equation} \label{c7-eqn:order-of-vanishing}
	r^{}_S(\chi)\, =\, \bigg(\sum_{v \in S} \dim^{}_{\CC}(V_{\chi}^{G_w}) \bigg) - 
	\dim^{}_{\CC}(V_{\chi}^G).
	\end{equation}
	Thus if either 
	$\chi$ is non-trivial and $S$ contains an (infinite) place $v$ such that $V_{\chi}^{G_{w}} \neq 0$
	or $\chi$ is trivial and $|S| > 1$ then the $\chi$-part of $\theta^{T}_{S}$ vanishes.
	Now suppose that $S$ contains all ramified primes.
	Then if  $\theta_{S}^{T}$ is non-trivial, precisely one of the following possibilities occurs:
	(i) $K$ is totally real and $L$ is totally complex,
	(ii) $K$ is an imaginary quadratic field,  $L/K$ is unramified and $S = S_{\infty}$ or 
	(iii) $L=K=\QQ$ and $S = S_{\infty}$. 
\end{remark}

\section{The abelian case}

In this section we assume that the extension $L/K$ is \emph{abelian}.
Let $\mu_L$ denote the roots of unity in $L$. It was independently shown in
\cite{c7-Ba78, c7-CN79, c7-DR80}
that one has
\begin{equation} \label{c7-eqn:integrality}
\Ann^{}_{\ZZ[G]}(\mu^{}_L) \theta^{}_S\, \subseteq\, \ZZ[G]
\end{equation}
whenever $S$ contains the set $S_{\ram}$ of 
all places of $K$ that ramify in $L/K$.
We now state Brumer's conjecture as discussed by Tate \cite{c7-Ta84}.
\begin{conjecture}[$B(L/K,S)$] \label{c7-conj:abelian-Brumer}
	\index{Brumer's conjecture!abelian} \index{class group}
	Let $L/K$ be an abelian extension of number fields 
	with Galois group $G$ and 
	let $S$ be a finite set
	of places of $K$ containing the set $S_{\infty}$ and all places of $K$ that ramify in $L/K$.
	Then one has
	\[
	\Ann^{}_{\ZZ[G]}(\mu^{}_L) \theta^{}_S\, \subseteq\, \Ann^{}_{\ZZ[G]}(\cl^{}_L).
	\]
\end{conjecture}

We will discuss relevant results on Brumer's conjecture 
and its generalisations in \S \ref{c7-sec:results} below. 
Here we only mention that
for absolutely abelian number fields Brumer's conjecture
follows from Stickelberger's theorem from the late 19th century.

\begin{theorem}[Stickelberger \cite{c7-St90}]
	\index{Stickelberger's theorem}
	Assume that $L$ is abelian over\/ $\QQ$. Then, Brumer's conjecture\/ 
	$B(L/K,S)$ holds.
\end{theorem}

\begin{remark}
	Consider the three cases of Remark \ref{c7-rem:vanishing}.
	In case (iii), Brumer's conjecture (and also the 
	Brumer--Stark conjecture below) is trivial. 
	Case (ii) follows from the fact that each ideal of $L$
	becomes principal in the Hilbert class field of $L$ 
	(see \cite[Remark 6.3]{c7-GP15}).
	Finally, case (i) can  be reduced to the case that $L$ is a CM-field 
	(see \cite[Proposition 6.4]{c7-GP15}). \exend
\end{remark}

Therefore, we shall henceforth assume that  $L/K$ is a CM-extension, 
\index{CM-extension}
that is, $L$ is a CM-field, $K$ is totally real and complex conjugation induces a unique automorphism
$j$ in $G$.

Let $w^{}_L$ be the cardinality of $\mu^{}_L$. Then clearly
$w^{}_L \in \Ann^{}_{\ZZ[G]}(\mu^{}_L)$ and so Brumer's conjecture asserts
that $\mathfrak{a}^{w^{}_L \theta^{}_S}$
is a principal ideal for every fractional ideal $\mathfrak{a}$ in $L$.
The Brumer--Stark conjecture now gives precise information
on a generator of this ideal. To make this precise,
we need the following definition.

\begin{definition}
	Let $L$ be a CM-field. Then, $\epsilon \in L^{\times}$
	is called an \emph{anti-unit} \index{anti-unit}
	if $\epsilon^{1+j} = 1$.
\end{definition}

\begin{conjecture}[$BS(L/K,S)$] \label{c7-conj:abelian-BS}
	\index{Brumer--Stark conjecture!abelian}
	Let $L/K$ be an abelian CM-extension with Galois group\/ $G$
	and let\/ $S$ be a finite set of places of\/ $K$ containing the set\/ 
	$S_{\infty}$ and all places of\/ $K$ that ramify in\/ $L/K$.
	Then, for every non-zero fractional ideal\/ $\mathfrak{a}$ in\/ $L$, there is an
	anti-unit\/ $\epsilon \in L^{\times}$ such that
	\begin{enumerate} \itemsep4pt
		\item 
		$\mathfrak{a}^{w_L \theta_S} = (\epsilon)$.
		\item
		The extension\/ $L(\sqrt[w_L]{\epsilon})/K$ is abelian.
	\end{enumerate}
\end{conjecture}

Here, the first part almost follows from $B(L/K,S)$, and Harold Stark originally suggested that the second condition might hold. 
The conjecture was first stated in published form by Tate \cite{c7-Ta84}.

\begin{remark}
	An anti-unit $\epsilon \in L^{\times}$ that fulfils (1)
	is determined up to a root of unity. \exend
\end{remark}

In order to generalise these conjectures to arbitrary Galois CM-extensions,
we need reformulations of both conjectures. For this, we introduce
the following terminology.
Let $S$ and $T$ be two finite sets of places of $K$.
We then say that $\hyp (S,T)$ is satisfied if the following holds:
\begin{itemize}
	\item
	$S$ contains all the archimedean places of $K$ and all places 
	which ramify in $L/K$,
	i.e.~$S \supseteq S_{\ram} \cup S_{\infty}$.
	\item
	$S \cap T = \emptyset$.
	\item
	$E_{L,S}^T$ is torsionfree.
\end{itemize}

\begin{remark}
	Note that $E_{L,S}^T$ is torsionfree whenever $T$ contains at least two places
	of different residue characteristic or at least one place of
	sufficiently large norm.  \exend
\end{remark}

\begin{lemma} \label{c7-lem:Coates}
	Let\/ $S$ be a finite set of places of\/ $K$ containing\/ $S_{\ram} \cup S_{\infty}$.
	Then, the elements\/ $1 - N(v)\phi_w^{-1}$, where\/ $v$ runs through all the 
	finite places of\/ $K$
	such that the sets\/ $S$ and\/ $T_{v} := \{v\}$ satisfy\/ $\hyp(S,T_{v})$, 
	generate\/ $\Ann_{\ZZ[G]} (\mu_L)$. Moreover,
	if we restrict to totally decomposed primes\/ $v$, the greatest common 
	divisor of the integers\/ $1 - N(v)$ equals\/ $w_L$.
\end{lemma}
\begin{proof}
	This is an obvious reformulation of 
	\cite[Ch. IV, Lemma 1.1]{c7-Ta84}.
\end{proof}

\begin{corollary} \label{c7-cor:Brumer-re}
	Let\/ $L/K$ be an abelian extension of number fields and let\/ $S$ be a finite set
	of places of\/ $K$ containing the set\/ $S_{\infty}$ and all places of\/ $K$ 
	that ramify in\/ $L/K$. Then, Brumer's conjecture\/ $B(L/K,S)$
	\index{Brumer's conjecture!abelian} holds if and only if
	\[
	\theta_S^T \in \Ann^{}_{\ZZ[G]}(\cl^{}_L)
	\]
	\index{Stickelberger element}
	for all finite sets\ $T$ of\/ $L$ such that\/ $\hyp(S,T)$ is
	satisfied.
\end{corollary}
\begin{proof}
	We have $\theta_S^T = \delta_T(0) \theta_S$ and $\delta_T(0) 
	\in \Ann_{\ZZ[G]}(\mu_L)$ whenever $\hyp(S,T)$ is satisfied.
	As $\delta_{T_v}(0) = 1 - N(v) \phi_w^{-1}$, the result follows
	from Lemma \ref{c7-lem:Coates}.
\end{proof}

Let $N_{L/K}: L^{\times} \rightarrow K^{\times}$ be the field-theoretic
norm map.
For $\epsilon \in L^{\times}$ we define
\[
S_{\epsilon} :=  
\{v \mbox{ finite place of } K: v \mid N_{L/K}(\epsilon) \}.
\]

\begin{proposition} \label{c7-prop:abelian-BS-re}
	Let\/ $S$ be a finite set of places of\/ $K$ containing\/ 
	$\mbox{$S_{\ram} \cup S_{\infty}$}$. Then, the Brumer--Stark conjecture\/ $BS(L/K,S)$
	\index{Brumer--Stark conjecture!abelian}
	is equivalent to the following claim:
	For every non-zero fractional ideal\/ $\mathfrak a$ of\/ $L$
	there is an anti-unit\/ $\epsilon \in L^{\times}$ such that
	\[
	\mathfrak a^{w^{}_L \cdot \theta^{}_S}\, =\, (\epsilon)
	\]
	and, for each finite set\/ $T$ of primes of\/ $K$ such that\/ 
	$\hyp(S \cup S_{\epsilon},T)$ is satisfied, 
	there is an\/ $\epsilon^{}_T \in E_{S_{\epsilon}}^T$
	such that
	\begin{equation} \label{c7-eqn:condition}
	\epsilon^{\delta_T(0)} = \epsilon_T^{w_L}.
	\end{equation}
\end{proposition}

\begin{proof}
	This is \cite[Ch. IV, Prop. 1.2]{c7-Ta84}.
\end{proof}

\begin{remark}
	In many cases, one can omit the condition that $\epsilon$
	is an anti-unit: \index{anti-unit}
	Suppose that the order of $\mathfrak a$ in the 
	class group is odd. Then, we may write
	$\mathfrak a = \mathfrak b^2 (u)$ for some $u \in L^{\times}$.
	Now assume that $\mathfrak b^{w_L \theta_S}$ is principal
	and generated by some $\beta \in L^{\times}$ such that
	\eqref{c7-eqn:condition} holds with $\epsilon$ replaced with
	$\beta$. As $(1-j) \theta_S = 2 \theta_S$, we then have
	$\mathfrak a^{w_L \theta_S} = (\epsilon)$, where
	$\epsilon := u^{w_L \theta_S} \cdot \beta^{1-j}$ is an
	appropriate anti-unit. \exend
\end{remark}

\section{The general case}

\subsection{How to generalise to non-abelian extensions?}

Now let $L/K$ be an arbitrary Galois CM-extension with Galois group $G$.
How can we formulate Brumer's conjecture and the Brumer--Stark conjecture
for this more general situation? 

We still have the Stickelberger elements $\theta_S^T$ at our disposal.
In view of Corollary \ref{c7-cor:Brumer-re} one is tempted to conjecture
that $\theta_S^T$ still annihilates the class group whenever
$\hyp(S,T)$ is satisfied. Although the Stickelberger elements
always belong to the centre of $\QQ[G]$, it is, however, in general not true
that $\theta_S^T$ has integral coefficients.
Thus $\theta_S^T$
does in general not even act on the class group!

\begin{example} \label{c7-ex:Nomura}
	This example is due to Nomura \cite[\S 6]{c7-No14b}.
	Let $\alpha$ be a root of $x^3-11x+7$ and set
	$L = \QQ(\sqrt{-3}, \sqrt{4001}, \alpha)$. Then, $L/\QQ$
	is a Galois CM-extension with Galois group $G$ isomorphic to
	$\ZZ/ 2 \ZZ \times S_3$, where $j$ generates the first factor
	and $S_3$ denotes the symmetric group \index{group!symmetric}
	on $3$ letters. We write
	\[
	S_3\, =\, \langle \sigma, \tau \mid \sigma^3 = \tau^2 = 1, \tau \sigma \tau^{-1} = \sigma^{-1} \rangle.
	\]
	Then $S_{\ram} = \left\{3, 4001 \right\}$ and
	for $S = S_{\ram} \cup S_{\infty}$ and $T = \left\{7\right\}$ one has
	\[
	\theta_S^T\, =\, \frac{1}{3}(1-j)\left( 3410 - 1774(\sigma + \sigma^2) + 44
	(\tau + \sigma \tau + \sigma^2 \tau) \right)
	\]
	\index{Stickelberger element}
	which visibly does not belong to $\ZZ[G]$. \exend
\end{example}

The idea is to replace the centre of $\ZZ[G]$ by a larger ring
$\mathcal{I}(G)$ such that $\theta_S^T$ always belongs to
$\mathcal{I}(G)$ and such that $\mathcal{I}(G) = \ZZ[G]$ when $G$ is abelian.
In order to achieve annihilators, one then has to multiply by a certain
`denominator ideal' $\mathcal{H}(G)$. We next introduce these purely algebraic
objects.

\subsection{Denominator ideals and the integrality ring}
\index{denominator ideal|(} \index{integrality ring|(}
Let $R$ be a Noetherian integrally closed domain with field of fractions $E$.
Let $A$ be a finite-dimensional separable $E$-algebra and let $\mathfrak{A}$ be an $R$-order in $A$. Our main examples are group rings
$\mathfrak{A} = R[G]$ and $A = E[G]$, 
where $R$ and $E$ are either $\ZZ$ and $\QQ$ or $\ZZ_p$ and $\QQ_p$ for a prime $p$,
respectively.
The reduced norm map \index{reduced norm}
$\nr = \nr_{A}: A \longrightarrow \zeta(A)$ is defined component-wise on the
Wedderburn decomposition of $A$ and extends to matrix rings over $A$ (see \cite[\S 7D]{c7-CR81}).
We choose a maximal $R$-order $\mathfrak{M}$ such that $\mathfrak{A} \subseteq \mathfrak{M} \subseteq A$.
Following \cite[\S 3.6]{c7-JN13}, for every matrix $H \in M_{n \times n} (\mathfrak{A})$ there is a generalised adjoint matrix
\index{generalised adjoint matrix}
$H^{\ast} \in M_{n\times n}(\mathfrak{M})$ such that $H^{\ast} H = H H^{\ast} = \nr (H) \cdot 1_{n \times n}$
(note that the conventions in \cite[\S 3.6]{c7-JN13} slightly differ from those in  \cite{c7-Ni10}).
If $\tilde{H} \in M_{n \times n} (\mathfrak{A})$ is a second matrix, then $(H \tilde{H})^{\ast} = \tilde{H}^{\ast} H^{\ast}$.
We define
\begin{align*}
\mathcal{H}(\mathfrak{A}) & := 
\{ x \in \zeta(\mathfrak{A}) \mid xH^{\ast} \in M_{n \times n}(\mathfrak{A}) \, \forall H \in M_{n \times n}(\mathfrak{A}) \, \forall n \in \NN \},\\
\mathcal{I}(\mathfrak{A}) & := 
\langle \nr(H) \mid H \in  M_{n \times n}(\mathfrak{A}), \,  n \in \NN\rangle_{\zeta(\mathfrak{A})}.
\end{align*}
One can show that these are $R$-lattices satisfying
\begin{equation}  \label{c7-eqn:HI_equals_H}
\mathcal{H}(\mathfrak{A}) \cdot \mathcal I(\mathfrak{A})\, =\, \mathcal{H}(\mathfrak{A}) 
\subseteq \zeta(\mathfrak{A}) \subseteq \mathcal{I}(\mathfrak{A}) \subseteq \zeta(\mathfrak{M}).
\end{equation}
Hence $\mathcal{H}(\mathfrak{A})$ is an ideal in the commutative $R$-order $\mathcal{I}(\mathfrak{A})$.
We will refer to $\mathcal{H}(\mathfrak{A})$ 
and $\mathcal{I}(\mathfrak A)$ as the \emph{denominator ideal} 
\index{denominator ideal}
and the \emph{integrality ring} \index{integrality ring}
of the $R$-order $\mathfrak{A}$, respectively.

\begin{remark}
	The integrality ring is the smallest subring of $\zeta(A)$ that contains
	$\zeta(\mathfrak A)$ and the image of the reduced norm of all matrices
	with entries in $\mathfrak A$. The denominator ideal measures the failure
	of the generalised adjoint matrices having coefficients in $\mathfrak A$. \exend	
\end{remark}

If $p$ is a prime and $G$ is a finite group, we set
\begin{align*}
\mathcal{I}(G) &:= \mathcal{I}(\ZZ[G]), &  \mathcal{I}_{p}(G) & := \mathcal{I}(\ZZ_{p}[G]), \\
\mathcal{H}(G) & := \mathcal{H}(\ZZ[G]), &  \mathcal{H}_{p}(G) & := \mathcal{H}(\ZZ_{p}[G]).
\end{align*}
The first claim of the following result is a special case of \cite[Prop.~4.4]{c7-JN13}.
The second claim then follows easily from \eqref{c7-eqn:HI_equals_H}.

\begin{proposition} \label{c7-prop:best-denominators}
	Let\/ $p$ be prime and $G$ be a finite group. Then, one has\/ 
	$\mathcal{H}_{p}(G) = \zeta(\ZZ_{p}[G])$ if and only if $p$ 
	does not divide the order of the commutator subgroup of $G$. Moreover, 
	in this case, we have\/ $\mathcal{I}_{p}(G) = \zeta(\ZZ_{p}[G])$.
\end{proposition}

Let $\mathfrak A = R[G]$, where $R$ is either $\ZZ$ or $\ZZ_p$ for a prime $p$.
As above let $\mathfrak M$ be a maximal order containing $\mathfrak A$. The central conductor \index{central conductor}
of $\mathfrak M$ over $\mathfrak A$
is defined to be
$\mathcal F(\mathfrak A) := \left\{x \in \zeta(\mathfrak M) |
x \mathfrak M \subseteq \mathfrak A \right\}$
and is explicitly given by (cf.~\cite[Thm.~27.13]{c7-CR81})
\begin{equation} \label{c7-eqn:conductor-formula}
\mathcal F(\mathfrak{A})\, =\, \bigoplus_{\chi} \frac{|G|}{\chi(1)}
\mathcal D^{-1} (E(\chi) / E),
\end{equation}
where $\mathcal D^{-1} (E(\chi)/ E)$ denotes the inverse different of
the extension 
\[
E(\chi)\, :=\, E(\chi(g): g \in G) \mbox{ over } E = \Quot(R)
\]
and the sum runs over all irreducible characters of $G$ 
modulo Galois action. It is clear from the definition that we
always have $\mathcal F(\mathfrak{A}) \subseteq \mathcal H(\mathfrak{A})$.
As above we set
\[
\mathcal{F}(G)\, :=\, \mathcal{F}(\ZZ[G]), \quad  \mathcal{F}_{p}(G)\, :=\, \mathcal{F}(\ZZ_{p}[G]).
\]

\begin{example} \label{c7-ex:dihedral}
	Let $p$ and $\ell$ be primes with $\ell$ odd. 
	We compute the denominator ideals of $\ZZ_{p}[D_{2\ell}]$, 
	where $D_{2 \ell}$ denotes the dihedral group of order $2 \ell$.
	\index{group!dihedral}
	In the case $\ell=3$, one has
	$D_{6} \simeq S_{3}$, the symmetric group on three letters.
	We let $\mathfrak{M}_{p}(D_{2\ell})$ be a maximal $\ZZ_p$-order
	containing $\ZZ_p[D_{2\ell}]$.
	Then
	\[\begin{array}{rl}
	\mathcal{H}_{p}(D_{2\ell})\negthickspace\negthickspace & =\, \begin{cases}
	\zeta(\ZZ_{p}[D_{2\ell}]), & \text{if $p\neq \ell$,} \\
	\mathcal{F}_{p}(D_{2\ell}), & \text{if $p=\ell$;}
	\end{cases} \\[5mm]
	\mathcal{I}_{p}(D_{2\ell})\negthickspace\negthickspace & =\, \begin{cases}
	\zeta(\ZZ_{p}[D_{2\ell}]), & \text{if $p \neq \ell$,} \\
	\zeta(\mathfrak{M}_{p}(D_{2\ell})) & \text{if $p=\ell$.}
	\end{cases} \end{array}
	\]
	In fact, in the case that $p \neq \ell$, 
	the result follows from Proposition \ref{c7-prop:best-denominators}.
	In the case $p=\ell$, the result is established in 
	\cite[Ex.~6]{c7-JN13}. \exend
\end{example}

\begin{example} \label{c7-ex:Aff(q)}
	Let $p$ be a prime and let $q = \ell^{n}$ be a prime power.
	We consider the group $\mathrm{Aff}(q) = \FF^{}_q \rtimes \FF_q^{\times}$
	of affine transformations 
	\index{group!of affine transformations}
	on $\FF_q$, the finite field 
	with $q$ elements.	
	Let $\mathfrak{M}_{p}(\Aff(q))$ be a maximal $\ZZ_{p}$-order such that $\ZZ_{p}[\Aff(q)] \subseteq \mathfrak{M}_{p}(\Aff(q)) \subseteq \QQ_{p}[\Aff(q)]$.
	Then by \cite[Prop.~6.7]{c7-JN16} we have
	\[\begin{array}{rl}
	\mathcal{H}^{}_{p}(\Aff(q)) \negthickspace\negthickspace & =\, 
	\begin{cases}
	\zeta(\ZZ^{}_{p}[\Aff(q)]), & \text{ if $p \neq \ell$,} \\
	\mathcal{F}^{}_{p}(\Aff(q)), & \text{ if $p=\ell \neq 2$;}
	\end{cases} \\[5mm]
	\mathcal{I}^{}_{p}(\Aff(q)) \negthickspace\negthickspace & =\, 
	\begin{cases}
	\zeta(\ZZ^{}_{p}[\Aff(q)]), & \text{if $p \neq \ell$,} \\
	\zeta(\mathfrak{M}^{}_{p}(\Aff(q))), & \text{if $p=\ell \neq 2$.}
	\end{cases} \end{array}
	\]
	If $p=\ell=2$, then we have containments
	\[\begin{array}{c}
	2\mathcal{H}_{2}(\Aff(q))\, \subseteq\, \mathcal{F}_{2}(\Aff(q))\, \subseteq\, \mathcal{H}_{2}(\Aff(q)), \\[2mm]
	%	\]
	%	\[
	2\zeta(\mathfrak{M}_{2}(\Aff(q)))\, \subseteq\, \mathcal{I}_{2}(\Aff(q))\, 
	\subseteq\, \zeta(\mathfrak{M}_{2}(\Aff(q))). \end{array}
	\]
	Note that the commutator subgroup of $\Aff(q)$ is $\FF_q$ so that
	the case $p \neq \ell$ again follows from Proposition
	\ref{c7-prop:best-denominators}.\exend
\end{example}

\begin{example}
	Let $S_4$ be the symmetric group on $4$ letters. 
	\index{group!symmetric}
	If $p$ is an odd prime, then 
	$\mathcal{I}_p(S_4) = \zeta(\mathfrak{M}_p(S_4))$
	and $\mathcal{H}_p(S_4) = \mathcal{F}_p(S_4)$. However, if $p=2$ we have
	\[\begin{array}{c}
	\quad 	\mathcal{F}^{}_2(S^{}_4)\, \subsetneq\, \mathcal{H}^{}_2(S^{}_4)\, 
	\subsetneq \zeta(\ZZ^{}_2[S^{}_4]); \\[2mm]
	\zeta(\ZZ^{}_2[S^{}_4])\ \subsetneq\  \mathcal{I}^{}_2(S^{}_4) \ 
	\subsetneq\, \zeta(\mathfrak M^{}_2(S^{}_4)). \end{array}
	\]
	This follows from \cite[Prop.~6.8]{c7-JN16}. \exend
\end{example}
\index{denominator ideal|)} \index{integrality ring|)}

\subsection{Integrality conjectures}
\index{integrality conjecture}
\index{Stickelberger element|(}
Assume that $L/K$ is a Galois CM-extension
with Galois group $G$. Let $S$ be a finite set of places of $K$
containing $S_{\ram} \cup S_{\infty}$.
We choose a maximal order $\mathfrak{M}(G)$ such that $\ZZ[G] \subseteq \mathfrak{M}(G) \subseteq \QQ[G]$.
If $T$ is a finite set of places of $K$ such that
$\hyp(S,T)$ is satisfied, then 
$\delta_T(0) = \nr(1 - N(v)\phi_w^{-1})$ belongs to $\zeta(\mathfrak{M}(G))$.
We define $\mathfrak{A}_{S}$ to be the $\zeta(\ZZ[G])$-submodule of $\zeta(\mathfrak{M}(G))$ generated
by the  elements $\delta_T(0)$, where $T$ runs through the finite sets of places of $K$
such that $\hyp(S,T)$ is satisfied.

\begin{conjecture} \label{c7-conj:integrality}
	Let $S$ be a finite set of places of $K$ containing $S_{\ram} \cup S_{\infty}$.
	Then, we have an inclusion
	\[
	\mathfrak{A}_S \theta_S\, \subseteq\, \mathcal{I}(G).
	\]
\end{conjecture}

As the integrality ring $\mathcal{I}(G)$ is always contained in the
centre of the maximal order $\mathfrak{M}(G)$, we may also state the
following considerably weaker conjecture.

\begin{conjecture} \label{c7-conj:weak-integrality}
	\index{integrality conjecture!weak}
	Let $S$ be a finite set of places of $K$ containing $S_{\ram} \cup S_{\infty}$.
	Then we have an inclusion
	\[
	\mathfrak{A}^{}_S \theta^{}_S\, \subseteq\, \zeta(\mathfrak{M}(G)).
	\]
\end{conjecture}

\begin{remark}
	If $p$ is a prime, we let $\mathfrak{M}_p(G) := \ZZ_p \otimes
	\mathfrak{M}(G)$ which is a maximal $\ZZ_p$-order in $\QQ_p[G]$.
	As we have
	\[
	\mathcal{I}(G)\, =\, \zeta(\QQ[G]) \cap \bigcap_p \mathcal{I}^{}_p(G)
	\mbox{ and }
	\zeta(\mathfrak{M}(G)) = \zeta(\QQ[G]) \cap \bigcap_p \zeta(\mathfrak{M}^{}_p(G))
	\]
	the integrality conjectures of this section naturally decompose
	into local conjectures at each prime $p$. \exend
\end{remark}

\begin{example}
	Consider the Galois CM-extension $L/\QQ$ from Example \ref{c7-ex:Nomura}.
	As before let $S = S_{\ram} \cup S_{\infty}$ and $T = \left\{7\right\}$.
	We have seen that in this case $\theta_S^T$ does not lie in
	$\zeta(\ZZ_3[G])$. However, one has
	\[
	\theta_S^T\, =\, \nr\left((1-j)\left(-\frac{71}{2} + \frac{1}{2}\sigma
	-11 \sigma^2 + 19 \tau + \frac{13}{2}\sigma \tau +
	\frac{37}{2}\sigma^2 \tau\right)\right) \in \mathcal{I}_3(G).
	\]
\end{example}

\begin{theorem} \label{c7-thm:abelian-int}
	Conjecture \ref{c7-conj:integrality} and Conjecture \ref{c7-conj:weak-integrality} both hold when $L/K$ is abelian.
\end{theorem}

\begin{proof}
	Lemma \ref{c7-lem:Coates} implies that $\mathfrak{A}_{S} = \Ann_{\ZZ[G]}(\mu_{L})$.
	Then, the result follows from \eqref{c7-eqn:integrality}
	and the fact that $\mathcal{I}(G) = \ZZ[G]$ in this case.
\end{proof}

Recall that a finite group is called \emph{monomial}
\index{group!monomial} if each of
its irreducible characters is induced by a linear character 
of a subgroup. The class of monomial groups includes all
nilpotent groups \cite[Thm.~11.3]{c7-CR81} and, more generally, 
all supersoluble groups \cite[Ch.~2, Cor.~3.5]{c7-We82}.

\begin{theorem} \label{c7-thm:maxint}
	Let\/ $L/K$ be a Galois extension with Galois group\/ $G \simeq H \times C$,
	where\/ $H$ is monomial and\/ $C$ is abelian. 
	Let\/ $S$ be a finite set of places of\/ $K$ containing\/ $S_{\ram} \cup S_{\infty}$.
	Then, we have an inclusion
	\[
	\mathfrak{A}^{}_S \theta^{}_S\, \subseteq\, \zeta(\mathfrak{M}(H))[C].
	\]
	In particular, Conjecture \textnormal{\ref{c7-conj:weak-integrality}}
	is true for monomial extensions.
\end{theorem}

\begin{proof}
	This is due to the author \cite[Thm.~1.2]{c7-Ni16}.
	The proof heavily relies on the abelian case and functoriality
	of Artin $L$-functions.
\end{proof}

For non-abelian extensions, unconditional results 
on Conjecture \ref{c7-conj:integrality} are 
rather sparse. Here we only mention the following
special case of \cite[Cor.~5.12]{c7-Ni16}.

\begin{corollary}
	Let\/ $\ell$ be an odd prime.
	Let\/ $L/K$ be a Galois CM-extension with Galois group isomorphic
	to $D_{4\ell}$, the dihedral group of order\/ $4 \ell$.
	\index{group!dihedral}
	Then, Conjecture \textnormal{\ref{c7-conj:integrality}} holds.
\end{corollary}

\begin{proof}
	We first note that $D_{4\ell} \simeq D_{2\ell} \times C_2$
	with $C_2 := \ZZ / 2\ZZ$ 
	and that dihedral groups are monomial.
	Taking Example \ref{c7-ex:dihedral} into account,
	we see that the $p$-part of Conjecture \ref{c7-conj:integrality}
	directly follows from Theorem \ref{c7-thm:maxint} if $p$ is odd.
	Now consider the case $p=2$.
	Let $N$ be the commutator subgroup of $D_{2\ell}$
	so that $D_{2\ell} / N \simeq C_2$. It follows
	from \cite[Proposition 2.13]{c7-JN16} that the group ring
	$\ZZ_2[D_{2\ell}]$ is `$N$-hybrid'
	\index{hybrid group ring} meaning that
	it decomposes into a direct product of
	$\ZZ_2[D_{2\ell}/N] \simeq \ZZ_2[C_2]$
	and some maximal order (see Definition
	\ref{c7-def:hybrid} below).
	As 
	\[
	\mathcal{I}^{}_2(D^{}_{4\ell})\, =\, \zeta(\ZZ^{}_2[D^{}_{4\ell}])\,
	=\, \zeta(\ZZ^{}_2[D^{}_{2\ell}])[C^{}_2],
	\]
	the result follows by combining
	Theorems \ref{c7-thm:maxint} and \ref{c7-thm:abelian-int}.
\end{proof}

We put $\omega_L := \nr (|\mu_L|) = \nr(w_L)$.

\begin{proposition} \label{c7-prop:omega-int}
	Let\/ $S$ be a finite set of places of\/ $K$ containing\/ $S_{\ram} \cup S_{\infty}$. 
	\begin{enumerate}
		\item 
		Suppose that Conjecture \textnormal{\ref{c7-conj:integrality}} holds. Then,
		\[
		\omega_L \theta_S \in \mathcal{I}(G).
		\]
		\item
		Suppose that Conjecture \textnormal{\ref{c7-conj:weak-integrality}} holds. 
		Then,
		\[
		\omega_L \theta_S \in \zeta(\mathfrak{M}(G)).
		\]
	\end{enumerate}
\end{proposition}

\begin{proof}
	It suffices to show that
	$\omega_L \theta_S \in \mathcal{A}_p(G)$ 
	for each prime $p$, where $\mathcal{A}_p(G) = \mathcal{I}_p(G)$
	in case (1) and $\mathcal{A}_p(G) = \zeta(\mathfrak{M}_p(G))$
	in case (2).
	By Lemma \ref{c7-lem:Coates}, there is a totally decomposed place
	$v_0$ of $K$ (in fact infinitely many places) such that 
	$|\mu_L| = (1-N(v_0)) \cdot c$, where $c$ is a unit in $\ZZ_p$,
	and such that $\hyp(S,T_0)$ is satisfied with $T_0 := \{v_0\}$.
	As $\nr(c)$ belongs to $\mathcal{A}_p(G)$, we have
	\[
	\omega_L \theta_S = \nr(c) \theta_S^{T_0} \in \mathcal{A}_p(G)
	\]
	as desired.
\end{proof}

\subsection{The non-abelian Brumer and Brumer--Stark conjectures} %\label{c7-subsec:non-abelian-Brumer}
The following conjecture was first formulated in \cite{c7-Ni11} and is a non-abelian generalisation of Brumer's Conjecture \ref{c7-conj:abelian-Brumer}.

\begin{conjecture}[$B(L/K,S)$] \label{c7-conj:Brumer}
	\index{Brumer's conjecture!non-abelian} \index{class group}
	Let\/ $S$ be a finite set of places of\/ $K$ containing\/ $S_{\ram} \cup S_{\infty}$.
	Then, for each\/ $x \in \mathcal H(G)$, we have
	\[
	x \cdot \mathfrak{A}^{}_S \theta^{}_S\, \subseteq\, \Ann^{}_{\zeta(\ZZ[G])} (\cl^{}_L).
	\]
\end{conjecture}

\begin{remark}
	Suppose that $G$ is abelian. As we have already observed,
	Lemma \ref{c7-lem:Coates} implies that 
	$\mathfrak{A}_{S} = \Ann_{\ZZ[G]}(\mu_{L})$ in this case.
	Since we have $\mathcal{H}(G) = \ZZ[G]$, Conjecture
	\ref{c7-conj:Brumer} recovers Brumer's Conjecture
	\ref{c7-conj:abelian-Brumer}. \exend
\end{remark}

\begin{remark}
	When Conjecture \ref{c7-conj:integrality} holds, then
	$x \cdot \mathfrak{A}_S \theta_S$ is at least contained in
	$\zeta(\ZZ[G])$ for each $x \in \mathcal H(G)$. \exend
\end{remark}

\begin{remark} %\label{c7-rmk:p-part-of-Brumer}
	If $M$ is a finitely generated $\ZZ$-module and $p$ is a prime, we define its $p$-part to be $M(p) := \ZZ_{p} \otimes M$.
	Replacing the class group $\cl_{L}$ by $\cl_{L}(p)$ for each prime $p$,
	Conjecture $B(L/K,S)$ naturally decomposes into local conjectures $B(L/K,S,p)$.
	It is then possible to replace $\mathcal{H}(G)$ by $\mathcal{H}_{p}(G)$ 
	(see \cite[Lemma 1.4]{c7-Ni11}).
	Moreover, if $p$ does not divide the order of the commutator subgroup of $G$ then 
	$\mathcal{H}_{p}(G) = \mathcal{I}_{p}(G) = \zeta(\ZZ_{p}[G])$ 
	by Proposition \ref{c7-prop:best-denominators} and so granting 
	the hypotheses on $S$ the statement of the local conjecture simplifies to
	\[
	\mathfrak{A}^{}_S \theta^{}_S\, \subseteq\, \Ann^{}_{\zeta(\ZZ^{}_{p}[G])} (\cl^{}_{L}(p)).
	\]
	%\vspace*{-4mm} \hfill \exend
\end{remark}
\vspace*{-4mm} \hfill \exend

\begin{remark}
	Burns \cite{c7-Bu11} has also formulated a conjecture which generalises many refined Stark conjectures to the
	non-abelian situation. In particular, it implies Conjecture \ref{c7-conj:Brumer} (see \cite[Prop.~3.5.1]{c7-Bu11}).
	A further approach to non-abelian Brumer and Brumer--Stark conjectures
	is due to Dejou and Roblot \cite{c7-DR14}.  \exend
\end{remark}

Recall that $\omega_L := \nr (|\mu_L|)$. The following is a non-abelian generalisation of the Brumer--Stark Conjecture \ref{c7-conj:abelian-BS}.

\begin{conjecture}[$BS(L/K,S)$] \label{c7-conj:Brumer--Stark}
	\index{Brumer--Stark conjecture!non-abelian}
	Let $S$ be a finite set of places of $K$ containing $S_{\ram} \cup S_{\infty}$.
	Then for each $x \in \mathcal H(G)$ we have
	$x \cdot \omega_L \cdot \theta_S \in \zeta(\ZZ[G])$.
	Moreover, for each non-zero fractional ideal $\mathfrak{a}$ of $L$,
	there is an anti-unit $\epsilon = \epsilon(x,\mathfrak{a},S) \in L^{\times}$ such that
	\[
	\mathfrak{a}^{x \cdot \omega_L \cdot \theta_S}\, =\, (\epsilon)
	\]
	and for each finite set $T$ of primes of $K$ such that $\hyp(S \cup S_{\epsilon},T)$ is satisfied there is an $\epsilon_{T} \in E_{L,S_{\epsilon}}^{T}$
	such that
	\begin{equation*} %\label{c7-eqn:abelian-ersatz}
	\epsilon^{z \cdot \delta_T(0)}\, =\, \epsilon_T^{z \cdot \omega_L}
	\end{equation*}
	for each $z \in \mathcal H(G)$.
\end{conjecture}

\begin{remark}
	If $G$ is abelian, we have 
	\[
	\mathcal{I}(G) = \mathcal{H}(G) = \ZZ[G] \quad \text{ and } \quad 
	\omega^{}_{L} = |\mu^{}_{L}| = w^{}_L.
	\]
	Hence it suffices to treat the case $x=z=1$ in this situation.
	Then, Proposition~\ref{c7-prop:abelian-BS-re} implies that
	Conjecture \ref{c7-conj:Brumer--Stark} generalises
	Conjecture \ref{c7-conj:abelian-BS}.  \exend
\end{remark}

\begin{remark}
	Suppose that Conjecture \ref{c7-conj:integrality} holds.
	Then, $\omega_L \theta_S \in \mathcal{I}(G)$ by Proposition 
	\ref{c7-prop:omega-int} and thus
	$x \cdot \omega_L \cdot \theta_S \in \zeta(\ZZ[G])$
	for each $x \in \mathcal H(G)$. \exend
\end{remark}

\begin{remark}
	When we restrict to ideals whose classes in $\cl_L$ have
	$p$-power order, we again obtain local conjectures 
	$BS(L/K,S,p)$ for each prime $p$.  \exend
\end{remark}

\subsection{The weak Brumer and Brumer--Stark conjectures}

Since $\mathcal H(G)$ always contains the central conductor $\mathcal F(G)$, we can state the following weaker 
versions of Conjectures $B(L/K,S)$ and $BS(L/K,S)$

\begin{conjecture}[$B_w(L/K,S)$] \label{c7-conj:Brumer-weak}
	\index{Brumer's conjecture!weak}
	Let $S$ be a finite set of places of $K$ containing $S_{\ram} \cup S_{\infty}$. Then, for each 
	$x \in \mathcal F(G)$ \index{central conductor}
	we have
	\[
	x \cdot \mathfrak A^{}_S \theta^{}_S\, \subseteq\, \Ann^{}_{\zeta(\ZZ[G])} (\cl^{}_L).
	\]
\end{conjecture}

\begin{conjecture}[$BS_w(L/K,S)$] \label{c7-conj:Brumer-Stark-weak}
	\index{Brumer--Stark conjecture!weak}
	Let $S$ be a finite set of places of $K$ containing $S_{\ram} \cup S_{\infty}$.
	Then for each $x \in \mathcal F(G)$ we have
	$x \cdot \omega_L \cdot \theta_S \in \zeta(\ZZ[G])$.
	Moreover, for each non-zero fractional ideal $\mathfrak a$ of $L$, 
	there is an anti-unit 
	$\epsilon = \epsilon(x,\mathfrak a,S) \in L^{\times}$ such that
	\[
	\mathfrak a^{x \cdot \omega_L \cdot \theta_S}\, =\, (\epsilon)
	\]
	and, for each finite set $T$ of primes of $K$ such that 
	$\hyp(S \cup S_{\epsilon},T)$ holds, there is an $\epsilon_T \in E_{S_{\epsilon}}^T$
	such that
	\[
	\epsilon^{z \cdot \delta_T(0)} = \epsilon_T^{z \cdot \omega_L}
	\]
	for each $z \in \mathcal F(G)$.
\end{conjecture}

\begin{remark}
	Suppose that Conjecture \ref{c7-conj:weak-integrality} holds.
	Then $x \cdot \mathfrak A_S \theta_S$ and 
	$x \cdot \omega_L \cdot \theta_S$ both lie in $\zeta(\ZZ[G])$
	(for the latter use Proposition \ref{c7-prop:omega-int}). \exend
\end{remark}

\begin{remark}
	We may again formulate local conjectures $B_w(L/K,S,p)$ and
	$BS_w(L/K,S,p)$ for each prime $p$.  \exend
\end{remark}

\subsection{Implications between the conjectures}

We first discuss dependence on the set $S$.

\begin{lemma} \label{c7-lem:dependence-S}
	Let\/ $S$ and\/ $S'$ be two finite sets of places of\/ $K$ such that\/
	$S$ contains\/ $S_{\ram} \cup S_{\infty}$.
	If $S \subseteq S'$, one has
	\begin{eqnarray*}
		B(L/K,S) & \implies & B(L/K,S') \\ 
		B_w(L/K,S) &  \implies & B_w(L/K,S')\\
		BS(L/K,S) & \implies & BS(L/K,S') \\ 
		BS_w(L/K,S) &  \implies & BS_w(L/K,S').
	\end{eqnarray*}
\end{lemma}

\begin{proof}
	We only give the proof in the case of Brumer's conjecture;
	the other cases follow similarly.
	So assume that $B(L/K, S)$ holds.
	We have $\theta_{S'} = \nr (\lambda) \theta_S$, where 
	$\lambda = \prod_{v \in S' \setminus S} (1 - \phi_w^{-1}) \in \ZZ[G]$.
	If $x$ lies in 
	$\mathcal H(G)$, so does $\tilde x := x \cdot \nr(\lambda)$.
	Hence we see that
	$x \mathfrak A_{S'} \theta_{S'} \subseteq \tilde x \mathfrak A_S
	\theta_S$ belongs to $\zeta(\ZZ[G])$ and annihilates $\cl_L$.
\end{proof}

The relation between the Brumer--Stark conjecture and Brumer's conjecture
is slightly more subtle.

\begin{lemma} \label{c7-lem:BS-implies-B}
	Let\/ $S$ be a finite set of places of\/ $K$ containing\/ $S_{\ram} \cup S_{\infty}$. Then,
	\begin{eqnarray*}
		BS(L/K,S) & \implies & B(L/K,S)\\
		BS_w(L/K,S) & \implies & B_w(L/K,S)
	\end{eqnarray*}
\end{lemma}

\begin{proof}
	We give the proof for the strong conjectures.
	Let $\mathfrak a$ be a non-zero fractional ideal of $L$ and
	let $x \in \mathcal H(G)$. Then
	$\mathfrak a^{x\cdot \omega_L \theta_S} = (\epsilon)$ and $(\epsilon)^{z\cdot \delta_T(0)} = (\epsilon_T)^{z \cdot \omega_L}$
	for all $z \in \mathcal H(G)$. Hence
	\begin{equation} \label{c7-eqn:fast}
	\mathfrak a^{x \cdot z \cdot \omega^{}_L \cdot \theta_S^T}\, =\, (\epsilon)^{z \cdot \delta^{}_T(0)} 
	= (\epsilon_T)^{z \cdot \omega_L}.
	\end{equation}
	Since $\omega_L \in \zeta(\QQ G)^{\times}$, we find $N \in \NN$ such that 
	$N \cdot \omega_L^{-1} \in \zeta(\ZZ[G])$.
	Moreover, $|G| \cdot \zeta(\ZZ[G]) \subseteq \mathcal F(G) \subseteq 
	\mathcal{H}(G)$ such that we may choose
	$z = |G| \cdot N \cdot \omega_L^{-1}$. However, the group of fractional ideals has no 
	$\ZZ$-torsion such that equation \eqref{c7-eqn:fast} implies that $x \cdot \theta_S^T$
	belongs to $\zeta(\ZZ[G])$ (take $\mathfrak{a}$ to be a 
	totally decomposed prime) and
	$\mathfrak a^{x \cdot \theta_S^T} = (\epsilon_T)$.
\end{proof}

\subsection{A strong Brumer--Stark property}

\begin{definition}
	\index{Brumer--Stark property!strong}
	\index{class group!ray class group}
	Let $S$ be a finite set of places of $K$ containing 
	$S_{\ram} \cup S_{\infty}$. We say that $L/K$ satisfies the
	\emph{strong Brumer--Stark property} $SBS(L/K,S)$ if 
	\begin{equation*} 
	\mathcal{H}(G) \cdot \frac{1}{2} \cdot \theta_S^T\, \subseteq\,
	\Ann^{}_{\zeta(\ZZ[G])} (\cl_L^T)
	\end{equation*}
	for all finite sets $T$ of $K$ such that $\hyp(S,T)$ holds.
\end{definition}

\begin{remark}
	It is clear that $SBS(L/K,S)$ holds if and only if
	\[
	\mathcal{H}^{}_p(G) \cdot \frac{1}{2} \cdot \theta_S^T\, \subseteq\,
	\Ann^{}_{\zeta(\ZZ^{}_p[G])} (\cl_L^T(p))
	\]
	for all primes $p$. It is then not hard to show that the
	strong Brumer--Stark property $StBS(L/K,S,p)$ 
	as discussed in \cite{c7-Ni11}
	for all $p$ implies our strong Brumer--Stark property $SBS(L/K,S)$. \exend
\end{remark}

\begin{remark}
	Suppose that $S'$ is a finite set of places of $K$ containing $S$.
	Then it follows as in the proof of Lemma \ref{c7-lem:dependence-S}
	that $SBS(L/K,S)$ implies $SBS(L/K,S')$.
\end{remark}

\begin{proposition} \label{c7-prop:SBS-BS}
	Let\/ $S$ be a finite set of places of\/ $K$ containing\/ 
	$S_{\ram} \cup S_{\infty}$. Then, $SBS(L/K,S)$ implies\/ $BS(L/K,S)$.
\end{proposition}

\begin{proof}
	The proof is similar to \cite[Prop.~3.9]{c7-Ni11}.
\end{proof}

\begin{remark}
	In order to prove $BS(L/K,S)$ and $B(L/K,S)$, it thus suffices to show
	that, for every prime $p$, we have:
	%	\begin{itemize}
	%		\item 
	%		$\frac{1}{2}\theta_S^T \in \mathcal{I}_p(G)$
	%		for all finite sets $T$ of $K$ such that $\hyp(S,T)$ holds;
	%		\item
	\[
	\mathcal{H}^{}_p (G) \cdot \frac{1}{2} \cdot \theta_S^T\, \subseteq\,
	\Ann^{}_{\zeta(\ZZ^{}_p[G])} (\cl_L^T(p))
	\]
	for all finite sets $T$ of $K$ such that $\hyp(S,T)$ holds.
	%	\end{itemize}
	We refer to this property as $SBS(L/K,S,p)$.
	This is in fact not much stronger than
	$BS(L/K,S,p)$. Nomura \cite[Proposition 4.2]{c7-No14b}
	showed that for odd $p$ in fact $SBS(L/K,S,p)$ and
	$BS(L/K,S,p)$ are equivalent.
\end{remark}

\begin{remark}
	Replacing the denominator ideal $\mathcal{H}(G)$ by the central conductor
	\index{central conductor}
	$\mathcal{F}(G)$ one can formulate a weaker variant
	$SBS_w(L/K,S)$ such that $SBS_w(L/K,S)$ implies $BS_w(L/K,S)$
	(we refrain from calling this the ``weak strong Brumer--Stark 
	property'' for obvious reasons). \exend
	\index{Brumer--Stark property!weak}
\end{remark}
\index{Stickelberger element|)}

\section{Relations to further conjectures and results} \label{c7-sec:results}

\subsection{The relation to the equivariant Tamagawa number conjecture}

We only give a vague description of the statement of the equivariant Tamagawa number 
conjecture (ETNC) 
\index{equivariant Tamagawa number conjecture}
for the relevant Tate motive \index{Tate motive}
as formulated by Burns and Flach \cite{c7-BF01}.

Let $L/K$ be a finite Galois extension of number fields with Galois group $G$.
We regard $h^{0}(\mathrm{Spec}(L))$ as a motive defined over $K$ and with coefficients in the semisimple algebra $\QQ[G]$.
Let  $\mathfrak{A}$ be a $\ZZ$-order such that $\ZZ[G] \subseteq \mathfrak{A} \subseteq \QQ[G]$.
The ETNC for the pair $(h^{0}(\mathrm{Spec}(L)), \mathfrak{A})$ asserts
that a certain canonical element $T\Omega(L/K, \mathfrak{A})$  of the relative 
algebraic $K$-group $K_{0}(\mathfrak{A},\RR)$ vanishes.
This element incorporates the leading coefficients of the Artin
$L$-functions attached to the irreducible characters of $G$ and certain
cohomological Euler characteristics.

We note that the ETNC for the pair $(h^{0}(\mathrm{Spec}(L)), \ZZ[G])$,
the Lifted Root Number Conjecture of Gruenberg, Ritter and Weiss 
\cite{c7-GRW99}, the vanishing of the element $T\Omega(L/K,0)$
defined in \cite[\S 2.1]{c7-Bu01}, 
and the `leading term conjecture at $s=0$' of \cite{c7-BB07} are all
equivalent (see \cite[Theorems 2.3.3 and 2.4.1]{c7-Bu01} and 
\cite[Remarks 4.3 and 4.5]{c7-BB07}).

Moreover, Burns and Flach \cite[\S 3, Cor.~1]{c7-BF03}
have shown that the ETNC for the pair $(h^{0}(\mathrm{Spec}(L)), \mathfrak{M}(G))$,
where $\mathfrak{M}(G)$ is a maximal order, is equivalent
to the strong Stark conjecture (as formulated by Chinburg 
\cite[Conj.~2.2]{c7-Ch83}) for $L/K$.
\index{strong Stark conjecture}

If $L/K$ is a Galois CM-extension, the ETNC 
(over the maximal order, and in general away from its $2$-primary part)
naturally decomposes into a plus and a minus part.
The following result is due to the author \cite[Thm.~3.5]{c7-Ni11}.

\begin{theorem} \label{c7-thm:SS-implies-BSw}
	Let\/ $L/K$ be a Galois CM-extension of number fields with Galois group\/ $G$.
	Let\/ $\mathfrak{M}(G)$ be a maximal order in\/ $\QQ[G]$ containing\/ $\ZZ[G]$.
	Then, the minus part of the 
	ETNC for the pair\/ $(h^{0}(\mathrm{Spec}(L)), \mathfrak{M}(G))$
	implies\/ $BS_w(L/K,S)$ and\/ $B_w(L/K,S)$ for all finite sets\/ $S$
	of places of\/ $K$ containing\/ $S_{\infty} \cup S_{\ram}$.
\end{theorem}

There is also a prime-by-prime version of Theorem
\ref{c7-thm:SS-implies-BSw} (see \cite[Thm.~4.1]{c7-Ni11}).
Combined with \cite[Cor.~2]{c7-Ni11b} this leads to the following
unconditional result (see also \cite[Cor.~4.2]{c7-Ni11}). 
We denote the maximal totally real subfield of $L$
by $L^+$ and let $L'$ be the Galois closure of $L$ over $\QQ$.
For a natural number $n$ we let $\zeta_n$ be a primitive $n$-th root of unity.

\begin{theorem}\label{c7-thm:weak-results}
	Let\/ $p$ be an odd prime. Assume that no prime of\/ $L^+$
	above\/ $p$ splits in\/ $L$ whenever\/ $L' \subseteq (L')^+(\zeta_p)$.
	Then, $BS_w(L/K, S, p)$ and\/ $B_w(L/K, S, p)$
	are true for every set\/ $S$ of places of\/ $K$ containing\/ 
	$S_{\ram} \cup S_{\infty}$. \qed
\end{theorem}

\begin{remark}
	We stress that for a given extension $L/K$ the hypotheses on $p$
	in Theorem \ref{c7-thm:weak-results}
	are fulfilled by all primes that do not ramify in $L'$.
	In particular, the hypotheses are satisfied by all but finitely many
	primes.  \exend
\end{remark}

\begin{theorem}\label{c7-thm:ETNC-implies-BS}
	Let\/ $L/K$ be a Galois CM-extension of number fields with Galois group\/ $G$.
	Let\/ $p$ be an odd prime. Then, the minus $p$-part of the 
	ETNC for the pair\/ $(h^{0}(\mathrm{Spec}(L)), \ZZ[G])$
	implies\/ $BS(L/K,S,p)$ and\/ $B(L/K,S,p)$ for all finite sets\/ $S$
	of places of\/ $K$ containing\/ $S_{\infty} \cup S_{\ram}$.
\end{theorem}
\begin{proof}
	When $\mu_L(p)$ is a cohomologically trivial $G$-module, this is
	due to Greither \cite{c7-Gr07} (if $G$ is abelian) and to the author
	\cite[Thm.~5.1]{c7-Ni11}; this includes the cases $\mu_L(p) = 1$
	and $p \nmid |G|$. The general case follows from recent
	work of Burns \cite[Proof of Cor.~3.11~(iii)]{c7-Bu17}.
\end{proof}
\begin{remark}
	The proofs of Theorem \ref{c7-thm:ETNC-implies-BS} actually show
	that a refinement of $SBS(L/K,S,p)$ holds 
	\index{Brumer--Stark property!strong}
	and then use an argument similar to
	Proposition \ref{c7-prop:SBS-BS}.  \exend
\end{remark}
There are meanwhile quite a few cases where the ETNC has been verified
for certain non-abelian extensions. Here we only mention
the following result of Johnston and the author \cite[Thm.~4.6]{c7-JN16}.

\begin{theorem} \label{c7-thm:Aff(q)}
	\index{group!of affine transformations}
	Let\/ $L/\QQ$ be a Galois extension with Galois group\/ $G \simeq \Aff(q)$,
	where\/ $q = \ell^n$ is a prime power. Then, the ETNC for the pair
	$(h^{0}(\mathrm{Spec}(L)), \mathfrak{M}(G))$ holds and the $p$-part
	of the ETNC for the pair\/
	$(h^{0}(\mathrm{Spec}(L)), \ZZ[G])$ holds for every prime\/ $p \neq \ell$.
\end{theorem}

\begin{remark}
	An extension $L/\QQ$ as in Theorem \ref{c7-thm:Aff(q)} never happens to
	be a CM-extension. However, Burns' conjecture on the annihilation 
	of class groups \cite{c7-Bu11} predicts non-trivial annihilators 
	for any Galois extension of number fields. Theorem \ref{c7-thm:Aff(q)}
	can then be combined with Example \ref{c7-ex:Aff(q)} 
	to show that Burns' conjecture
	holds in this case (up to a factor $2$ if $\ell = 2$).
	This is \cite[Thm.~7.6]{c7-JN16}.  \exend
\end{remark}

\begin{remark}
	The $\ell$-part of the ETNC in the situation of Theorem 
	\ref{c7-thm:Aff(q)} is considered in recent work
	with Henri Johnston \cite{c7-JN17c}.
	Suppose in addition that $L$ is totally real and 
	Leopoldt's conjecture holds for $L$ at $\ell$.
	Then, the ETNC for the pair
	$(h^{0}(\mathrm{Spec}(L))(1), \ZZ[G])$ holds.
	Moreover, the ETNC for the pair
	$(h^{0}(\mathrm{Spec}(L)), \ZZ[G])$ holds if $\ell$ is at most
	tamely ramified (see \cite[Cor.~10.6]{c7-JN17c}).
	For the proof one has to verify the `$\ell$-adic Stark
	conjecture at $s=1$' for $L/\QQ$ which might be seen as an analogue
	at $s=1$ of Gross' conjecture \ref{c7-conj:Gross-II} below.
\end{remark}

\subsection{$p$-adic Artin $L$-functions}

Let $p$ be an odd prime and let $K$ be a totally real field.
Let $\mathcal{L}/K$ be a Galois extension of $K$ such that $\mathcal{L}$ is totally real and contains the cyclotomic $\ZZ_{p}$-extension $K_{\infty}$ of
$K$ and $[\mathcal{L} : K_{\infty}]$ is finite.
We put $\mathcal{G} := \Gal(\mathcal{L} / K)$ and 
$\Gamma_K := \Gal(K_{\infty} / K) \simeq \ZZ_p$ such that
$\mathcal G \simeq H \rtimes \Gamma$, where 
$H := \Gal(\mathcal{L}/ K_{\infty})$ and $\Gamma = \mathcal{G}/H \simeq
\Gamma_K$. Thus $\mathcal{L} / K$ is a one-dimensional 
$p$-adic Lie extension. 
\index{p-adic Lie extension}
We choose a topological generator $\gamma_K$ of $\Gamma_K$.

We write $\chi_{\mathrm{cyc}}$ for the $p$-adic cyclotomic character
\index{character!cyclotomic}
\[
\chi_{\mathrm{cyc}}:\, \Gal(\mathcal{L}(\zeta_{p})/K) \longrightarrow \ZZ_{p}^{\times},
\]
defined by $\sigma(\zeta) = \zeta^{\chi_{\mathrm{cyc}}(\sigma)}$ for any $\sigma \in \Gal(\mathcal{L}(\zeta_{p})/K)$ and any $p$-power root of unity $\zeta$.
Let $\omega$ and $\kappa$ denote the composition of $\chi_{\mathrm{cyc}}$ with the projections onto the first and second factors of the canonical decomposition $\ZZ_{p}^{\times} = \mu^{}_{\QQ^{}_p} \times (1+p\ZZ^{}_{p})$, respectively;
thus $\omega$ is the Teichm\"{u}ller character.
\index{character!Teichm\"{u}ller}
We note that $\kappa$ factors through $\Gamma_{K}$ and
put $u := \kappa(\gamma_{K})$.

Fix a character $\psi \in \Irr_{\CC_{p}}(\mathcal{G})$
and let $S$ be a finite set of places of $K$ containing all
archimedean places and all places that ramify in $\mathcal{L}/K$.
Note that $S$ in particular  contains the set $S_p$
of all $p$-adic places.
Each topological generator $\gamma_{K}$ of  $\Gamma_{K}$ permits the definition of a 
power series $G^{}_{\psi,S}(T) \in \QQ_{p}^{c} \otimes^{}_{\QQ^{}_{p}} 
Quot(\ZZ^{}_{p} [[T]])$ 
by starting out from the Deligne--Ribet power series 
\index{Deligne--Ribet power series}
for one-dimensional characters of open subgroups 
of $\mathcal{G}$ (see \cite{c7-CN79, c7-DR80, c7-Ba78}) 
and then extending to the general case by using Brauer induction (see \cite{c7-Gr83}).
One then has an equality \index{Artin $L$-function!p-adic}
\[
L^{}_{p,S}(1-s,\psi) = \frac{G^{}_{\psi,S}(u^s-1)}{H^{}_{\psi}(u^s-1)},
\]
where $L_{p,S}(s,\psi): \ZZ_p \rightarrow \CC_p$ 
denotes the `$S$-truncated $p$-adic Artin $L$-function'
(a $p$-adic meromorphic function) attached to $\psi$ 
constructed by Greenberg \cite{c7-Gr83},
and where, for irreducible $\psi$, one has
\[
H_{\psi}(T)\, =\, 
\begin{cases}
\psi(\gamma_{K})(1+T)-1, & \text{if $H \subseteq \ker \psi$,} \\
1, & \text{otherwise.} 
\end{cases}
\]

\subsection{The interpolation property and Gross' conjecture}

Let $p$ be an odd prime and
choose a field isomorphism $\iota: \CC \simeq \CC_p$.
For a character $\psi \in \Irr_{\CC_p}(\mathcal G)$ we put
$\psi^{\iota} := \iota^{-1} \circ \psi \in \Irr_{\CC}(\mathcal G)$.
If $\psi$ is a linear character and $r \geq 1$ is an integer then for every choice of field isomorphism $\iota: \CC \simeq \CC_p$ one
has the interpolation property
\index{interpolation property}
\begin{equation} \label{c7-eqn:interpolation-property}
L_{p,S}(1-r, \psi) =
\iota\left(L_{S}(1-r,(\psi\omega^{-r})^{\iota})\right).
\end{equation}
This can be extended to characters $\psi$ of arbitrary degree provided
that $r \geq 2$ (see \cite[\S 4]{c7-Gr83}). However, the argument fails
in the case $r=1$. Nevertheless, it seems plausible to
conjecture the following.

\begin{conjecture} \label{c7-conj:int-property}
	For each\/ $\psi \in \Irr_{\CC_{p}}(\mathcal{G})$, one has
	\[ \label{c7-eqn:int-at-0}
	L^{}_{p,S}(0, \psi)\, =\,
	\iota\left(L^{}_{S}(0,(\psi\omega^{-1})^{\iota})\right).
	\]	
\end{conjecture}

\begin{remark}
	As both sides in \eqref{c7-eqn:int-at-0}
	are well-behaved with respect to direct sum, inflation and induction of characters, it is easy to see that
	Conjecture \ref{c7-conj:int-property} holds
	when $\psi$ is a monomial character \index{character!monomial}
	(also see the discussion in \cite[\S 2]{c7-Gr81}).  \exend
	%	From recent work of Burns \cite[Theorem 5.2 (i)]{c7-Bu17}
	%	and of Spiess ??? it follows that the left hand side 
	%	of \eqref{c7-eqn:int-at-0}
	%	vanishes whenever the right hand side does.
\end{remark}

In fact, Conjecture \ref{c7-conj:int-property} is a special case
of a Conjecture of Gross \cite{c7-Gr81} which we now recall.
Let $\chi \in \Irr_{\CC}(\Gal(\mathcal L(\zeta_p)/K))$ 
be a non-trivial character and let $r_{S}(\chi)$ be 
the order of vanishing of the $S$-truncated Artin
$L$-function $L_S(s,\chi)$ at $s=0$. We write
$L_S^{\ast}(0, \chi)$ for the leading coefficient in the Laurent series
expansion of $L_S(s,\chi)$ at $s=0$.
Now let $\psi \in \Irr_{\CC_{p}}(\mathcal{G})$ and choose a field
isomorphism $\iota: \CC \simeq \CC_p$. Then, formula \eqref{c7-eqn:order-of-vanishing}
shows that 
\[
r_S(\psi)\, :=\, r_S((\psi\omega^{-1})^{\iota})
\]
does in fact not depend on the choice of $\iota$.
The first part of Gross' conjecture \cite{c7-Gr81} concerns
the order of vanishing of $p$-adic Artin $L$-functions
and asserts the following.

\begin{conjecture} \label{c7-conj:Gross-I}
	\index{Gross' conjecture!first part}
	Let\/ $\psi \in \Irr_{\CC_{p}}(\mathcal{G})$. Then, the order of
	vanishing of the\/ $S$-truncated\/ $p$-adic Artin\/ $L$-function\/
	$L_{p,S}(s, \psi)$ at\/ $s=0$ equals\/ $r_S(\psi)$. 
\end{conjecture}

\begin{remark}
	Suppose that $\psi$ is linear. If $r_S(\psi)$ vanishes, 
	Conjec\-ture \ref{c7-conj:Gross-I} holds by \eqref{c7-eqn:interpolation-property}. 
	The conjecture is also known when $r_S(\psi) = 1$ (see
	\cite[Prop. 2.13]{c7-Gr81}). \exend
\end{remark}

\begin{theorem}
	\index{order of vanishing}
	Let\/ $\psi \in \Irr_{\CC_{p}}(\mathcal{G})$. Then, the order of
	vanishing of the $S$-truncated\/ $p$-adic Artin\/ $L$-function\/
	$L_{p,S}(s, \psi)$ at\/ $s=0$ is at least\/ $r_S(\psi)$. 
\end{theorem}

\begin{proof}
	For $\psi$ being a linear character, this has been proved by Spiess \cite{c7-Sp14}
	using Shintani cocycles (his approach actually allows $p$
	to be equal to $2$). The general case has recently been settled by 
	Burns \cite[Thm.~3.1]{c7-Bu17}. 
\end{proof}

The rest of this subsection is mainly devoted to (the second part of)
Gross' conjecture. This may be skipped by the reader who is only
interested in the Brumer and Brumer--Stark conjectures.

Fix a character $\psi \in \Irr_{\CC_{p}}(\mathcal{G})$ and choose
a Galois CM-extension $L$ over $K$ such that $\psi \omega^{-1}$ factors
through $G := \Gal(L/K)$. 
We denote the kernel of the natural augmentation map $\ZZ[S(L)] 
\rightarrow \ZZ$ that maps each $w \in S(L)$ to $1$ by $X_{L,S}$.
The usual Dirichlet map \index{Dirichlet map}
\begin{eqnarray*}
	\lambda_{L,S}:\, \RR \otimes E_{L,S} & \rightarrow & \RR \otimes X_{L,S}\\
	1 \otimes \epsilon & \mapsto & \sum_{w \in S(L)} \log  |\epsilon |_w w
\end{eqnarray*}
is an isomorphism of $\RR[G]$-modules. For each place $w$ of $L$,
Gross \cite[\S 1]{c7-Gr81} defines a $p$-adic absolute value
\[
\lVert \cdot \rVert_{w,p}:\, L^{\times} \rightarrow \ZZ_p^{\times}
\]
as the composite map
\[
L^{\times} \hookrightarrow L_w^{\times} \rightarrow 
\Gal(L_w^{\mathrm{ab}} / L_w) \rightarrow
\ZZ_p^{\times};
\]
here, $L_w^{\mathrm{ab}}$ denotes the maximal abelian extension of $L_w$,
the first arrow is the natural inclusion, the second arrow is the 
reciprocity map of local class field theory and the last map
is the $p$-adic cyclotomic character.
We define a homomorphism of $\ZZ_p[G]$-modules
\begin{eqnarray*}
	\lambda^{}_{p,L,S}:\, \ZZ^{}_p \otimes E^{}_{L,S} & \rightarrow & 
	\ZZ^{}_p \otimes X^{}_{L,S}\\
	1 \otimes \epsilon & \mapsto & 
	\sum_{w \in S(L)} \log  \lVert \epsilon \rVert^{}_{w,p} w.
\end{eqnarray*}
Now, choose a field isomorphism $\iota: \CC \simeq \CC^{}_p$. Then, 
$\lambda_{L,S}^{-1}$ and $\lambda^{}_{p,L,S}$ induce an endomorphism
\[
(\CC^{}_p \otimes^{}_{\ZZ^{}_p} \lambda^{}_{p,L,S}) \circ 
(\CC^{}_p \otimes^{}_{\iota} \lambda_{L,S}^{-1}):\, \CC^{}_p \otimes X^{}_{L,S}
\rightarrow \CC^{}_p \otimes X^{}_{L,S}.
\]
We define a $p$-adic regulator \index{regulator}
\[
R_{p, S}^{(\iota)}(\psi)\, :=\, \mathrm{det}^{}_{\CC_p}((\CC^{}_p \otimes^{}_{\ZZ^{}_p} 
\lambda^{}_{p,L,S}) \circ 
(\CC^{}_p \otimes^{}_{\iota} \lambda_{L,S}^{-1}) \mid
\Hom^{}_{\CC^{}_p[G]}(V^{}_{\psi \omega^{-1}}, \CC^{}_p \otimes X^{}_{L,S})).
\]

\begin{proposition} \label{c7-prop:vanishing-re}
	Fix\/ $\psi \in \Irr_{\CC_{p}}(\mathcal{G})$ and choose a field isomorphism\/ 
	$\iota: \CC \simeq \CC_p$. Then Conjecture~\textnormal{\ref{c7-conj:Gross-I}} 
	holds for\/ $\psi$ if and only if\/ $R_{p, S}^{(\iota)}(\psi) \neq0$.
\end{proposition}

\begin{proof}
	This follows from \cite[Thm.~3.1~(iii)]{c7-Bu17}.
\end{proof}

For an integer $r$, let $L_{p,S}^r(0, \psi)$ be the coefficient of $s^r$
in the power series expansion of $L_{p,S}(s, \psi)$ at $s=0$.
We can now state the second part of Gross' conjecture.

\begin{conjecture} \label{c7-conj:Gross-II}
	\index{Gross' conjecture!second part}
	Fix\/ $\psi \in \Irr^{}_{\CC^{}_{p}}(\mathcal{G})$ and choose a field 
	isomorphism\/ $\iota: \CC \simeq \CC_p$. Then, one has
	\[
	L_{p,S}^{r_S(\psi)}(0, \psi)\, =\, R_{p,S}^{(\iota)}(\psi) \cdot
	\iota\left(L_S^{\ast}(0,(\psi \omega^{-1})^{\iota})\right).
	\]
\end{conjecture}

\begin{remark}
	By means of Proposition \ref{c7-prop:vanishing-re}, it is clear that
	Conjecture \ref{c7-conj:Gross-II} is only interesting when
	Conjecture \ref{c7-conj:Gross-I} holds.  \exend
\end{remark}

The following recent result due to 
Dasgupta, Kakde and Ventullo \cite{c7-DKV17} generalises the approach
developed in \cite{c7-DDP11}.

\begin{theorem} \label{c7-thm:Gross-linear}
	Conjecture~\textnormal{\ref{c7-conj:Gross-II}} holds for linear characters. \qed
\end{theorem}

\begin{corollary}
	Suppose that Conjecture~\textnormal{\ref{c7-conj:Gross-I}} holds for all\/
	$\psi \in \Irr_{\CC_p}(\mathcal{G})$. Then, Conjecture~\textnormal{\ref{c7-conj:Gross-II}} is also true for all\/ $\psi \in \Irr_{\CC_p}(\mathcal{G})$.
\end{corollary}

\begin{proof}
	This follows from Theorem \ref{c7-thm:Gross-linear}
	by Brauer induction.
\end{proof}

\subsection{Conditional Results}

We can now state the following result which has been proved 
by Johnston and the author
\cite[Thm.~5.2 and Cor.~5.4]{c7-JN17b}.
We refer to the `equivariant Iwasawa main conjecture' (EIMC)
\index{Iwasawa main conjecture}
for totally real fields (see \cite{c7-Ka13} or 
\cite{c7-RW11}, for instance).

\begin{theorem} \label{c7-thm:Brumer-Stark}
	Let\/ $L/K$ be a Galois CM-extension with Galois group\/ $G$ and
	let\/ $p$ be an odd prime. Suppose that the EIMC for the extension\/
	$L(\zeta_p)_{\infty}^+ / K$ holds. Suppose further that
	Conjecture~\textnormal{\ref{c7-conj:int-property}} holds for all irreducible
	characters of\/ $\Gal(L(\zeta_p)_{\infty}^+ / K)$ which factor
	through\/ $\Gal(L^+/K)$. Then, $SBS(L/K,S,p)$ and thus\/
	$BS(L/K,S,p)$ and\/ $B(L/K,S,p)$ are true for every finite set\/ $S$
	of places of\/ $K$ containing\/ $S_p \cup S_{\ram} \cup S_{\infty}$.
\end{theorem}

\begin{remark}	
	\index{$\mu$-invariant}
	We write $\mu_p(F)$ for the $p$-adic $\mu$-invariant attached to
	the cyclotomic $\ZZ_p$-extension of a number field $F$ 
	(see \cite[Rem.~4.3]{c7-JN17a} for details).
	When $\mu_p(L(\zeta_p)^+)$ vanishes, then the EIMC
	has been proved independently by Kakde
	\cite{c7-Ka13} and Ritter and Weiss	\cite{c7-RW11}.
	Without assuming the vanishing of $\mu$-invariants
	considerable progress has been made in
	\cite{c7-JN17a}. This includes the case $p \nmid |G|$. \exend
\end{remark}

\begin{remark}	
	Suppose that $L/K$ is abelian.
	Then, Conjecture~\ref{c7-conj:int-property} holds by
	\eqref{c7-eqn:interpolation-property}.
	Under the somewhat stronger condition that $\mu$ vanishes,
	Theorem~\ref{c7-thm:Brumer-Stark} has been shown
	by Greither and Popescu \cite{c7-GP15} by an entirely different method.
	This method has been generalised to arbitrary Galois CM-extensions
	by the author in \cite{c7-Ni13}.  \exend
\end{remark}

In order to get rid of the $p$-adic places one has to assume the full
strength of Gross' conjecture.

\begin{theorem}
	\index{equivariant Tamagawa number conjecture}
	Let\/ $L/K$ be a Galois CM-extension with Galois group\/ $G$ and
	let\/ $p$ be an odd prime. Suppose that\/ $\mu_p(L^+)$ vanishes
	or that\/ $p \nmid |G|$. Suppose further that Gross' Conjecture
	\textnormal{\ref{c7-conj:Gross-I}} holds for all\/ $\psi \in 
	\Irr_{\CC_p}(\Gal(L^+_{\infty} / K))$. Then, the minus\/ $p$-part
	of the ETNC for the pair\/ $h^0(\mathrm{Spec}(L), \ZZ[G])$ holds.
	In particular, both\/ $BS(L/K,S,p)$ and\/ $B(L/K,S,p)$ are true for
	all finite sets\/ $S$ of places of\/ $K$ containing\/
	$S_{\ram} \cup S_{\infty}$.
\end{theorem}

\begin{proof}
	This is \cite[Cors.~3.8 and 3.11]{c7-Bu17}. Note that the last part
	follows from Theorem \ref{c7-thm:ETNC-implies-BS}.
\end{proof}

\subsection{Unconditional results}
We now discuss certain cases where the Brumer--Stark conjecture
holds unconditionally. Let $p$ be a prime and let $G$ be a finite group.
For a normal subgroup $N \unlhd G$, let $e^{}_{N} = |N|^{-1}\sum_{\sigma \in N} \sigma$
be the associated central trace idempotent in the group algebra $\QQ_{p}[G]$.

\begin{definition} \label{c7-def:hybrid}
	\index{hybrid group ring}
	Let $N \unlhd G$. We say that the $p$-adic group ring $\ZZ_{p}[G]$ is 
	\emph{$N$-hybrid} if (i) $e_{N} \in \ZZ_{p}[G]$ (i.e. $p \nmid |N|$) and (ii) 
	$\ZZ_{p}[G](1-e_{N})$ is a maximal $\ZZ_{p}$-order in $\QQ_{p}[G](1-e_{N})$.
\end{definition}

\begin{theorem}\label{c7-thm:unconditional-BS}
	Let\/ $L/K$ be a finite Galois CM-extension of number fields.
	Let\/ $N$ be a normal subgroup of\/ $G:=\Gal(L^{+}/K)$ and let\/ $F=(L^{+})^{N}$.
	Let\/ $p$ be an odd prime and let\/ $\overline{P}$ be a Sylow\/ $p$-subgroup 
	of\/ $\overline{G}:=\Gal(F/K) \simeq G/N$.
	Suppose that\/ $\ZZ_{p}[G]$ is\/ $N$-hybrid, $G$ is monomial, 
	\index{group!monomial} and\/ $F^{\overline{P}}/\QQ$ is abelian.
	Let\/ $S$ be a finite set of places of\/ $K$ such that\/ 
	$S_{p} \cup S_{\ram}(L/K) \cup S_{\infty} \subseteq S$.
	Then, both $BS(L/K,S,p)$ and\/ $B(L/K,S,p)$ are true.
\end{theorem}

\begin{proof}
	It follows from the theory of hybrid Iwasawa algebras
	\cite{c7-JN17a} that the relevant case of the EIMC holds.
	We also recall that Conjecture \ref{c7-conj:int-property} holds
	for monomial characters.
	Then the result follows from Theorem \ref{c7-thm:Brumer-Stark}.
	See \cite[Thm.~10.5]{c7-JN17b} for details.
\end{proof}

We recall that a \emph{Frobenius group}
\index{group!Frobenius}
is a finite group $G$ 
with a proper nontrivial subgroup $V$
such that $V\cap gVg^{-1}=\{ 1 \}$ for all $g \in G-V$,
in which case $V$ is called a \emph{Frobenius complement}.
A Frobenius group $G$ contains a unique normal subgroup $U$, 
known as the \emph{Frobenius kernel}, such that
$G$ is a semidirect product $U \rtimes V$. 

\begin{corollary}\label{c7-cor:BS-Frob}
	Let\/ $L/K$ be a finite Galois CM-extension of number fields and 
	let\/ $G=\Gal(L^{+}/K)$.
	Suppose that\/ $G = U \rtimes V$ is a Frobenius group with Frobenius 
	kernel\/ $U$ and abelian Frobenius complement\/ $V$.
	Suppose further that\/ $(L^{+})^{U}/\QQ$ is abelian (in particular, this 
	is the case when\/ $K=\QQ$).
	Let\/ $p$ be an odd prime and let\/ $S$ be a finite set of places of\/ $K$ 
	such that\/ $S_{p} \cup S_{\ram}(L/K) \cup S_{\infty} \subseteq S$.
	Suppose that either\/ $p \nmid |U|$ or\/ $U$ is a\/ $p$-group (in particular, 
	this is the case if\/ $U$ is an\/ $\ell$-group for any prime\/ $\ell$.)
	Then, both\/ $BS(L/K,S,p)$ and\/ $B(L/K,S,p)$ are true.
\end{corollary}

\begin{proof}
	This is \cite[Cor.~10.7]{c7-JN17b}. We recall the proof for convenience.
	First note that $G$ is monomial by \cite[Lemma 9.7]{c7-JN17b} since $V$ is abelian.
	Suppose that $p \nmid |U|$. Let $N=U$ and $F=(L^{+})^{N}$. Then $\ZZ_{p}[G]$ 
	is $N$-hybrid by \cite[Prop.~9.4]{c7-JN17b}.
	Hence, the desired result follows from Theorem \ref{c7-thm:unconditional-BS} 
	in this case since $F/\QQ$ is abelian, which forces $F^{\overline{P}}/\QQ$ to be abelian.
	Suppose that $U$ is a $p$-group. Taking $N=\{1\}$ and $F=L^{+}$, we apply
	Theorem \ref{c7-thm:unconditional-BS} with $\overline{G}=G$ and $\overline{P}=U$ to obtain the desired result. 
\end{proof}

\begin{example}
	In particular, $U$ is an $\ell$-group in Corollary \ref{c7-cor:BS-Frob}  
	when $G$ is one of the following Frobenius groups
	(for a natural number $n$ we denote the cyclic
	group of order $n$ by $C_n$):
	\begin{itemize}
		\item \index{group!of affine transformations}
		$G \simeq \mathrm{Aff}(q) = \FF_q \rtimes \FF_q^{\times}$,
		where $q$ is a prime power and
		$\mathrm{Aff}(q)$ is the group of affine transformations on $\FF_q$,
		\item $G \simeq C_{\ell} \rtimes C_{q}$, where $q<\ell$ are distinct primes such that $q \mid (\ell-1)$ and $C_{q}$ acts on $C_{\ell}$ via an embedding $C_{q} \hookrightarrow \Aut(C_{\ell})$,
		\item $G$ is isomorphic to any of the Frobe\-nius groups con\-struct\-ed 
		in  \cite[Ex.~2.11]{c7-JN17a}. 
	\end{itemize} 
	Note that in particular $\mathrm{Aff}(3) \simeq S_3$
	is the symmetric group on $3$ letters \index{group!symmetric}
	(which is the smallest non-abelian group) and
	$\mathrm{Aff}(4) \simeq A_4$ is the alternating group
	on $4$ letters.
	\exend
\end{example}

In certain situations, we can also remove the condition that 
$S_{p} \subseteq S$. 
To illustrate this, we conclude with the following two results
(the first is \cite[Thm.~10.10]{c7-JN17b}, whereas the second 
is due to Nomura \cite{c7-No14a, c7-No14b}).

\begin{theorem}
	Let\/ $L/\QQ$ be a finite Galois CM-extension of the rationals.
	Suppose that\/ $\Gal(L / \QQ) \simeq \langle j \rangle \times G$, where\/ 
	$G =\Gal(L^{+}/\QQ) = N \rtimes V$
	is a Frobenius group with Frobenius kernel $N$ and abelian Frobenius 
	complement\/ $V$.
	Suppose further that\/ $N$ is an\/ $\ell$-group for some prime\/ $\ell$.
	Then, both\/ $BS(L/\QQ,S,p)$ and\/ $B(L/\QQ,S,p)$ are true for every 
	odd prime\/ $p$ and every finite set\/ $S$ of places of\/ $\QQ$ such 
	that\/ $S_{\ram}(L/\QQ) \cup S_{\infty} \subseteq S$. \qed
\end{theorem}

\begin{theorem}
	\index{group!dihedral}
	Let\/ $\ell$ be an odd prime. Let\/ $L/K$ be a Galois CM-extension 
	with Galois group\/ $G \simeq D_{4 \ell}$. Let\/ $p$ be a prime and 
	suppose that\/ $p$ does not split in\/ $\QQ(\zeta_{\ell})$. Then,
	$BS(L/K,S,p)$ and\/ $B(L/K,S,p)$ both hold for every finite set\/ $S$ 
	of places of\/ $K$ such that\/ $ S_{\ram} \cup S_{\infty} \subseteq S$. \qed
\end{theorem}

%A figure has the following format:0

%\begin{figure}[t]
%\begin{center}
%  \includegraphics[width=0.5\textwidth]{file}
%\end{center}
%\caption{\label{c7-project-vorde} Figure caption}
%\end{figure}

%All labels must start with the project number, to make
%them distinguishable.

\end{document}